\documentclass{amsart}
\usepackage{latexsym} 
\usepackage{amsmath,amssymb,amsfonts,amsthm}
\usepackage{verbatim}
\usepackage{textcomp}

\newtheorem{remark}{Remark}
\newtheorem{theorem}{Theorem}
\newtheorem{proposition}{Proposition}
\newtheorem{corollary}{Corollary}
\newtheorem{definition}{Definition}
\newtheorem{lemma}{Lemma}
\newtheorem{example}{Example}

\newcommand{\mP }{\mathbb{P}}
\newcommand{\me}{\mathbb{E}}
\newcommand{\mn}{\mathbb{N}}
\newcommand{\mr }{ \mathbb{R}}

\newcommand{\s}{\textbf{\textrm{S}}}

\title{\textbf{Lyapunov 1-forms for diffusion Processes}}
\author{Diego S. Ledesma}
\begin{document}
\begin{abstract}
We extend the concept of Lyapunov 1-forms for the case of diffusion processes to study its asymptotic behavior. We give some examples and a condition for the existence of these objects.
\end{abstract}

\address{Departamento de
 Matem\'{a}tica, Universidade Estadual de Campinas,\\ 13.081-970 -
 Campinas - SP, Brazil.}
\email{dledesma@ime.unicamp.br}
\maketitle

\section{Introduction}

In this article we deal with geometry and stochastic processes. In particular we are interested to extend the concept
of Lyapunov 1-forms introduced (until the author knowledge) by Farber et. al. in \cite{Farber}. In that article the
authors focus their attention in dynamical systems and obtain a result on the existence of Lyapunov 1-forms and a
generalization of Conley Index Theorem. To do this they introduce a cohomological operator whose properties
joint to some other dynamical properties (chain recurrence) determine an existence result.

Here, the author is interested in generalize the concept of Lyapunov 1-form for diffusion processes. The idea is the following, in a compact manifold $M$ consider a diffusion $X$ and an invariant set $W$. Then a Lypaunov 1-form $\omega$ will be a closed 1-form in $M\backslash W$ such that the Stratonovich integral satisfy
\[
 \me\left[\int \omega ~\delta X\right]<0,
\]
for $X(0)\in M\backslash W$. This will let us to study the asymptotic behavior of the paths of the diffusion.

The main obstacle to work in the same way that Farber et. al. is that there are no pretty good (for this set up) definitions for chain
recurrence on diffusions. So we will adapt an idea of Lastchev \cite{Latschev} to obtain an existence result.

There are several approaches to define the integral of 1-forms along stochastic processes (see, e.g., Manabe \cite{Manabe} and Manabe et. al. \cite{Manabe2}) we will use the approach of the second order geometry introduced by Schwartz in the early 80's, which it seems to be
the natural set up to do stochastic calculus on manifolds. The best introduction to the subject is the book of Emery
\cite{emery}.

In section 2 we define the operator the cohomological operator $  \mathcal{J}^L_\mu$ in terms of the diffusion
invariant
measures $\mu$ and the diffusor $L$ and study its properties. We give special attention to diffusions with infinitesimal
generator of the form $L=\frac{1}{2}\Delta_M+V$ with $V$ a vector field. Also we relate $ \mathcal{J}^L_\mu$ with some Homology
cycles which, following S. Schwartzman \cite{SolS}, we may call them asymptotic cycles. Finally in section 3 we introduce
Lyapunov 1-forms, we give an existence result (Theorem 1) and some examples.

\section{Asymptotic Cycles}

In this work $(M,g)$ will be a compact Riemannian manifold and $X$ be a Stochastic process in
$M$. Let $L$ be a semi elliptic second order differential operator in $M$ without constant term and with coeficients in $C^2(M)$. We say that the stochastic process
$X$ is an $L$-diffusion if the real process \footnote{The operator $L$ is the generator of the Markov semigroup
associated to $X$}
\[
 f(X)-f(X_0)-\int (Lf)(X)~dt
\]
is a local martingale for any $f\in C^\infty(M)$.

Consider a $L$-diffusion $X$ in $M$ and define the subset $\mathcal{M}_L$ of the measure space $\mathcal{M}$ as
follows: $\mu\in\mathcal{M}_L$ if
\begin{eqnarray}\label{condition}
 \int_M(Lf)~\mu=0
\end{eqnarray}
for any function $f\in C^\infty(M)$. It is easy to see that $\mathcal{M}_L$ is non empty because, since $M$ is compact,
there are invariant measures and it is an easy exercise to see that invariant measures satisfy the equation
(\ref{condition}) for any $f$ (see, for example, Ikeda and Watanabe \cite{IW}).

Let $\s:\Omega^1(M)\rightarrow \Gamma(\tau^* M)$  be the Stratonovich operator (see \cite[Theorem 7.1]{emery}). The Stratonovich Integral of a $C^\infty-$one form $\alpha$ with respect to a semimartingale $X$ (see \cite[Definition 7.3]{emery}) is given by
\[
 \int\alpha~\delta  X=\int \s\alpha~\mathbb{D}X.
\]
\begin{example} Let $X$ be the solution of
\[
 dX=V(X)~dt+\sum_{i=1}^n X_i(X)~\delta W^i
\]
 for $X(0)=x$, where $(W_1,\ldots,W_n)$ is a Brownian motion in $\mr^n$ and $V,X_1,\ldots, X_n\in\Gamma(TM)$. Then, the infinitesimal generator is $L=V+\sum_{i=1}^nX_i^2$ where and for a  $C^\infty-$one form $\alpha$ we have that
\[
\s\alpha(L)=\alpha(V)+\sum_{i=1}^nX_i\alpha(X_i)
\]
therefore, the Stratonovich integral of $\alpha$ with respect to $X$ is
\[
 \int\alpha~\delta  X=\int \s\alpha(L)~(X)~dt+\sum_{i=1}^n\int \alpha(X_i)(X)~dW^i.
\]
where $\delta$ denote the Stratonovich differential.
\end{example}

For a measure $\mu$ in $\mathcal{M}_L$ define the functional $ \mathcal{J}^L_\mu$  over the space of $C^\infty-$one forms $\Omega^1(M)$ as
\[
  \mathcal{J}^L_\mu(\alpha)=\int_M\s\alpha(L)~\mu
\]
for $\alpha$ in $\Omega^1(M)$. This definition is based on Farber article \cite{Farber} but the idea can be tracked back to Schwartzman \cite{SolS}.

\begin{lemma}
 The operator $ \mathcal{J}^L_\mu$ has the following properties
\begin{itemize}
 \item[i-] $ \mathcal{J}^L_\mu(\alpha+df)= \mathcal{J}^L_\mu(\alpha)$,
 \item[ii-] $ \mathcal{J}^L_\mu(\lambda\alpha+\beta)=\lambda \mathcal{J}^L_\mu(\alpha)+ \mathcal{J}^L_\mu(\beta)$,
\end{itemize}
for any $\alpha,~\beta$ in $\Omega^1(M)$, $f\in C^\infty(M)$ and $\lambda\in \mr$.
\end{lemma}
\begin{proof}
  We just prove item i- the other one is almost trivial. Since $\mu$ is in $\mathcal{M}_L$ and $\s(\alpha+
df)=\s\alpha+\s df$ we get
that
\begin{eqnarray*}
 \int_M (\s\alpha+\s df) (L)~\mu&=&\int_M\s\alpha(L)~\mu+\int_M (Lf)~\mu\\
&=&0+\int_M\s\alpha(L)~\mu.
\end{eqnarray*}

Thus
\[
  \mathcal{J}^L_\mu(\alpha+df)= \mathcal{J}^L_\mu(\alpha)
\]
\end{proof}

Thus $ \mathcal{J}^L_\mu$ induce a linear functional on the first cohomology group
\[ \mathcal{J}^L_\mu: H^1(M,\mr)\rightarrow \mr\] by $ \mathcal{J}^L_\mu(\xi)= \mathcal{J}^L_\mu(\alpha)$, for any closed
1-form
$\alpha\in\xi$. The
Homological Poincare duality implies that there is a 1-cycle $\rho_\mu^L\in H_1(M,\mr)$ such that
\[
  \mathcal{J}^L_\mu(\alpha)=<\alpha,\rho_\mu^L>
\]
for any closed 1-form $\alpha$.

\begin{corollary}
Let $X$ be a L-diffusion $\mu\in \mathcal{M}_L$. If the associated cycle $\rho_\mu^L$ is the null homology class then
$ \mathcal{J}^L_\mu(\alpha)=0$ for any closed 1-form $\alpha$.
\end{corollary}

In particular we can consider diffusions with $L=\frac{1}{2}\Delta_M+V$ for a $V\in\Gamma(TM)$ then
\[
 \mathcal{M}_L=\{\mu=k\phi~\mu_g,~k\in\mr\}
\]
for some smooth positive function $\phi$ (see, for example, Ikeda and Watanabe \cite[Proposition 4.5]{IW}). Using this
result we get the following:

\begin{proposition}\label{p1}
Let $X$ be a diffusion with infinitesimal generator $L=\frac{1}{2}\Delta_M+V$ then
\begin{eqnarray*}
  \mathcal{J}^L_{\mu}(\alpha)&=&\frac{1}{2}\int_Mg(\alpha^\sharp,2V-\nabla \ln(\phi))~\mu\\
&=&\frac{1}{2}\int_Mg(\alpha,2V^\flat-d \ln(\phi))~\mu
\end{eqnarray*}
for all $\alpha\in\Omega^1(M)$ and $\mu\in\mathcal{M}_L$. Here $\sharp$ denotes the natural isomorphism
$^\sharp:\Omega^1(M)\rightarrow \Gamma(TM)$ and $\flat$
denotes the natural isomorphism $^\flat:\Gamma(TM)\rightarrow \Omega^1(M)$.
\end{proposition}
\begin{proof}
 Any 1-form $\alpha$ in $M$ can be written as
\[
 \alpha=\sum_{i=1}^Nf_i~dh_i
\]
for some smooth functions $\{f_i,h_i\}_{i=1}^N$ and $N$ big enough. Thus
\begin{eqnarray*}
  \s\alpha\left(\frac{1}{2}\Delta_M+\nabla f\right)
&=&\frac{1}{2}\sum_{i=1}^N(f_i~\Delta_M h_i+g(\nabla f_i,\nabla h_i)+2f_i~g(\nabla h_i,\ V)\\
&=&\frac{1}{2}\sum_{i=1}^N(\textrm{div}(f_i~\nabla h_i)+2f_i~g(\nabla h_i,V))\\
&=&\frac{1}{2}\textrm{div}(\alpha^\sharp)+g(\alpha^\sharp,V)
\end{eqnarray*}
where $\alpha^\sharp=\sum_{i=1}^Nf_i~\nabla h_i$.

Then, using that $\mu=k\phi~\mu_g$, the properties of the divergence operator and the divergence theorem we get
that
\begin{eqnarray*}
  \mathcal{J}^L_{\mu}(\alpha)
&=&\frac{k}{2}\int_M(\textrm{div}(\alpha^\sharp)+2~g(\alpha^\sharp,V))\phi~\mu_g\\
&=&\frac{k}{2}\int_M\Big(\textrm{div}(\phi\alpha^\sharp)-g(\alpha^\sharp,\nabla
\phi)+2\phi~g(\alpha^\sharp,V)\Big)~\mu_g\\
&=&\frac{k}{2}\int_Mg(\alpha^\sharp,2\phi~V-\nabla \phi)~\mu_g\\
&=&\frac{1}{2}\int_Mg(\alpha^\sharp,2V-\nabla \ln(\phi))~\mu
\end{eqnarray*}
\end{proof}

\begin{remark}\label{r1}
 The function $\phi$ is such that the de-Rham-- Kodaira decomposition of the  1-form $2\phi V^\flat$  is given by
\[
 \phi V^\flat=\frac{1}{2}d\phi+\delta\beta+\gamma.
\]
 with $\gamma$ a harmonic in $\Omega^1(M)$ and $\beta$ in $\Omega^2(M)$ (see, for example, Ikeda and Watanabe \cite[pg. 294]{IW}). Also, from the proof of the proposition above, the de
Rham--Kodaira decomposition of $\phi V^\flat$ and some facts of differential geometry we get that
\[
  \mathcal{J}^L_\mu(\alpha)=k\int_Mg(\alpha,\gamma)~\mu_g=k<\alpha,\gamma>_{\mu_g},
\]
for any $\alpha\in[\alpha]\in H^1(M,\mr)$.
\end{remark}

When $L=\frac{1}{2}\Delta_M+\nabla F$ for a smooth function $F$ it is known that $\phi=\exp(2F)$ (see,
for example, Ikeda and Watanabe \cite[pg. 294]{IW}).

\begin{corollary}
Let $X$ be a diffusion with infinitesimal generator $L=\frac{1}{2}\Delta_M+\nabla F$, where $F$ is a smooth, then
we have that $ \mathcal{J}^L_{\mu}\equiv0$ for all $\mu\in\mathcal{M}_L$.
\end{corollary}
In general $ \mathcal{J}^L_{\mu}$ is not always zero, for example when $L=\frac{1}{2}\Delta_M+V$ with $V$ a harmonic vector field.

The following result will be useful to characterize $L$ diffusions and to handle with the second order operator $L$.
\begin{lemma}\label{caracteriza-difusiones}
 A stochastic process $X$ is an $L$-diffusion if and only for  any 1-form $\alpha$
\[\int \alpha~\delta X=\int \s\alpha(L)(X)~dt+\textrm{local martingale}.\]
\end{lemma}
\begin{proof}
 Since any 1-form $\alpha\in\Omega^1(M)$ can be written as
\[
 \alpha=\sum_ia_i df_i
\]
we get, by the properties of the stochastic integral, that
\begin{eqnarray*}
 \int \alpha~\delta X&=&\int \s\alpha~\mathbb{D}X\\
&=&\sum_{i}\int a_i~\s df_i~\mathbb{D}X\\
&=&\sum_{i}\int a_id\left(\int\s df_i~\mathbb{D}X\right)\\
&=&\sum_{i}\int a_id\left(\int Lf_i(X)~dt\right)+\textrm{local martingale}\\
&=&\sum_{i}\int a_i\int Lf_i(X)~dt+\textrm{local martingale}\\
&=&\int\s\alpha(L)(X)~dt+\textrm{local martingale}
\end{eqnarray*}
On the other side, assume that $\int \alpha ~\delta X=\int \s\alpha(L)(X)~dt+\textrm{local martingale}$ for any 1-form $\alpha\in \Omega^1(M)$. Taking $\alpha=df$ we get that
\[
 f(X)-f(X_0)=\int Lf(X)~dt+\textrm{local martingale}.
\]
Thus $X$ is an $L$-diffusion.
\end{proof}


In \cite{ochi} Y. Ochi associates to a diffusion $X$ with infinitesimal generator $L=\frac{1}{2}\Delta+V$ a semimartingale $\tilde X^\lambda$, $0\leq\lambda<\infty$, on the dual space of $\Omega^1(M)$ by
\[
 \tilde X_t^{\lambda}(\alpha)=\frac{1}{\sqrt{\lambda}}\int_0^{\lambda t}\alpha~\delta X.
\]
to study ocupation time laws and homogenization theorems.
Using a result of M. Arnaudon in \cite[pp. 253]{arnaudon} we observe that the martingale part and the bounded variation part of $X^\lambda_t$ can be written, using Schwartz geometry, as
\[
 \frac{1}{\sqrt{\lambda}}\int_0^{\lambda t}\alpha ~dX^m\hspace{1cm} A_t^\lambda(\alpha)=\frac{1}{\sqrt{\lambda}}\int_0^{\lambda t}\s \alpha ~DX
\]
where $dX^m$ is a 1-vector and $DX$ is a 2-vector such that $\mathbb{D}X=dX^m+DX$ and given, in local coordinates $(x_1,\ldots,x_n)$, by
\[
 dX^m=dM^i\partial_i\hspace{2cm} DX=dA^i\partial_i+\frac{1}{2}d[M^i,M^j]\partial_{ij}
\]
where $M^i$ and $A^i$ are the martingale and bounded variation part of the real process $x_i(X)$ \cite[pp. 254]{arnaudon}.

The Lemma \ref{caracteriza-difusiones}, joint to the calculations on the proof of Proposition \ref{p1}, imply that the bounded variation part $A_t^\lambda$ of $X_t^\lambda$ is characterized by
\[
 A_t^\lambda(\alpha)=\frac{1}{\sqrt{\lambda}}\int_0^{\lambda t} \s\alpha(L)(X)~dt=\frac{1}{\sqrt{\lambda}}\int_0^{\lambda t} \left(\frac{1}{2}\textrm{div}(\alpha^\sharp)+g(\alpha^\sharp,V)\right)(X)~dt
\] as Y. Ochi observes in \cite[pp. 255]{ochi}.
Also, he proves that $Y_t^\lambda$ (Doob decomposition) have continuous versions that converges in probability, as $\lambda\rightarrow\infty$, to a Wiener process $Z$ with mean functional $0$ and covariance functional $<\alpha,\beta>(t\wedge s)$.

If we assume that the measure $\mu$ is an invariant ergodic measure, then by the ergodic theorem and the Lemma 2 we have that
\begin{eqnarray*}
  \mathcal{J}^L_\mu(\alpha)
&=&\int_M\s\alpha(L)~\mu\\
&=&\lim_{t\rightarrow\infty}\frac{1}{t}\int_0^t\s\alpha(L)(X)~ds\hspace{2cm} \mu-\textrm{a.e.}\\
&=&\lim_{t\rightarrow\infty}\frac{1}{t}\int_0^t\s\alpha~\mathbb{D}X\hspace{2,9cm} \mu-\textrm{a.e.}\\
&=&\lim_{t\rightarrow\infty}\frac{1}{t}\int_0^t\alpha~\delta X\hspace{3,2cm} \mu-\textrm{a.e.}
\end{eqnarray*}
where $\delta$ denotes the Stratonovich differential (see \cite[Proposition 7.4]{emery}). The calculations above 
prove the following proposition.

\begin{proposition}\label{p2}
 Let $\mu$ be an invariant ergodic measure. Then there is a 1-cycle $\rho_\mu^L\in H_1(M,\mr)$ such that for any closed 1-form $\alpha$ we have that
\[
\lim_{t\rightarrow\infty}\frac{1}{t}\int_0^t\alpha~\delta X= <[\alpha],\rho_\mu^L>\hspace{2cm} \mu-\textrm{a.e.}
\]
where $[\alpha]$ is the associated cohomology class in $ H^1(M,\mr)$
\end{proposition}

Therefore, when $L=\frac{1}{2}\Delta_M+V$ e $\mu$ is the only ergodic measure associated to $L$, we get that
\[
 \mathcal{J}^L_\mu(\alpha)=\lim_{\lambda\rightarrow\infty}\frac{A_1^\lambda(\alpha)}{\sqrt{\lambda}}.
\]

\section{Lyapunov One forms}

Given a $L$-diffusion $X$ we can define the transition probability function $P_t(x,U)$ for a measurable set $U$, a real
$t>0$ and a point $x$ as follows
\[
 P_t(x,U)=\me[I_U(X_t^x)]=\mP\left[X_t^x\in U\right].
\]

We say that a subset $U$ of $M$ is invariant under $X$ if for any $x\in U$ we have that $P_t(x,U)=1$ for all $t\geq0$.

\begin{definition}\label{def-lyapunov-form}
 Let $X$ be an $L$-diffusion in $M$ and $U$ be a closed subset in $M$ invariant under $X$. A closed 1-form $\beta$ in
$\Omega^1(M\backslash U)$ is called a Lyapunov 1-form for $(X,U)$ if
$\s\beta(L)<0$, in $M\backslash U$.
\end{definition}

Given a Lyapunov 1-form for $(X,U)$, consider the function $f:[0,\infty]\times (M\backslash U)\rightarrow \mr$ defined by
\[
 f(t,x)=\me\left[\int_0^t\s\beta(L)(X_s^x)ds\right]
\]
It is easy to see that $\partial_tf(0,x)=\s\beta(L)(x)<0$. Also, using Lemma \ref{caracteriza-difusiones} and the definition of Stratonovich integral, we get
\begin{eqnarray*}
 f(t,x)&=&\me\left[\int_0^t\s\beta(L)(X_s^x)ds\right]\\
&=&\me\left[\int_0^t \s\beta(L)~\mathbb{D}X_s^x\right]\\
&=&\me\left[\int_0^t\beta~\delta X_s^x\right]
\end{eqnarray*}
Therefore, in principle, we can use $f(x,t)$ to study the behavior of the paths of $X$ for small $t>0$ as we see in the following examples.



\begin{example}
 Let $S^{n}\subset \mr^{n+1}$ the sphere and $V$ be the vector field over $S^n$ given by the gradient of the height function $h(x_1,\ldots,x_{n+1})=x_1$, this is
\[
 V(x)=(1,0,\ldots,0)-x_1x
\]
for any $x\in S^n$. Consider the stochastic differential equation $dx_t=V(x_t)~\circ dB_t$ for $B_t$ the $\mr$-Brownian motion. Using the It\^o formula we get that is infinitesimal generator is given by $L=\frac{1}{2}V^2$.

To determine $\mathcal{M}_L$ we first observe that for any $f:S^n\rightarrow \mr$ we have that
\[
 Lf(x_1,\vec{y})=(1-x_1^2)\left((1-x_1^2)\partial_1^2f-2x_1(\partial_1f+\vec{y}\cdot d(\partial_1f))+\vec{y}\cdot df\right)+x_1^2\textrm{Hess}(f)(\vec{y},\vec{y}).
\]
By this formula it is easy to see that $Lf(N)=Lf(S)=0$ for $N=(1,0,\ldots 0)$ and $S=(-1,0,\ldots 0)$. Consider $f_1,f_2:S^n\rightarrow \mr$ given by $f_1(x)=\frac{1}{2}x_1^2$  and $f_2(x)=x_1$, then a simple calculation shields that
$$Lf_1(x)=(1+3x_1^2)(1-x_1^2)\hspace{1cm}\textrm{and}\hspace{1cm}Lf_2(x)=-2x_1(1-x_1^2).$$
Therefore the only measures such in $\mathcal{M}_L$ are convex combination of $\{\delta_N,~\delta_S\}$.

Consider the invariant set given by $W=\{N,S\}$. Let $f=\ln(1-x_1^2)$ be a function in $S^n\backslash W$ then $Lf(x)=-2(1-x_1^2)$ and therefore $\omega=df$ defines a Lyapunov one form for $x_t$. We observe that
\[
 \me[f(x_t)]=f(x)-2\int_0^t\me[1-x_1(s)^2]~ds
\]
implies
\[
 f(x)-2t\leq\me[f(x_t)].
\]
Fix $k>2t$ and let $A_k$ be a subset of $\Omega$ given by
\[
 A_k=\{f(x_t)\leq f(x)-k\}
\]
then
\begin{eqnarray*}
 f(x)-2t&\leq&\int_{A_k}f(x_t)~\mP+\int_{A^C}f(x_t)~\mP\\
&\leq&(f(x)-k)\mP(A_k)
\end{eqnarray*}
therefore
\[
 \mP(A_k)\leq\frac{2t-f(x)}{k-f(x)}.
\]

\end{example}
\begin{example}
Let $T^2=([0,1]\times [0,1])/\sim$ where $\sim$ is the equivalent relation identifying $(x,0)\sim (x,1)$ and $(0,y)~(1,y)$. Let $V$ the vector field $V$ given, in local coordinates, by $ V=\sin(2\pi x)\partial_x+\cos(2\pi x)\partial_y$. Consider the following stochastic differential equation
\[
dz_t=V(z_t)~\circ dB_t.
\]
The It\^o formula implies that the infinitesimal generator is given by $L=\frac{1}{2} V^2$. Therefore, the measures $\mu$ given by
\[
 \mu=\lambda \delta_{\{0\}\times S^1}+(1-\lambda) \delta_{\{1/2\}\times S^1}\hspace{.5cm} 0\leq\lambda\leq1
\]
where $S^1\simeq (\{k\}\times [0,1])/\sim$, are contained in $\mathcal{M}_L$. Consider the invariant set given by $W=(\{1/2\}\times [0,1])/\sim $ and the  1-form defined by $\beta=dy/\sim$. Clearly
\[
 \s\beta(L)=-\pi\sin(2\pi x)^2\leq0~\Rightarrow~ \int_{W^C} \s\beta(L)~\delta_{(\{0\}\times [0,1])/\sim}=0
\]
and there are not Lyapunov 1-forms. If $W=(\{1/2\}\times [0,1])/\sim\cup (\{0\}\times [0,1])/\sim$   then $\beta=dy/\sim$ is a Lyapunov 1-form and
\[
 \me\left[\int_0^t\beta~\delta z_t\right]=-\pi\me\left[\int_0^t(\sin(2\pi z_s)^2~ds\right]<0.
\] 

Also, if we consider the function $f:T^2\backslash W\rightarrow \mr$ given by $f(x,y)=\ln(\sin(2\pi x)^2)$ then $Lf(x,y)=-2\pi^2\sin^2(2\pi x)$ and
\[
 \me[f(z_t)]=f(z_0)-2\pi^2\int_0^t\me[(\sin(2\pi x_s))^2]~ds<f(z_0).
\] As before, fix $k>2\pi^2$ and let $A_k$ be a subset of $\Omega$ given by
\[
 A_k=\{f(x_t)\leq f(x)-k\}
\]
then
\[
 \mP(A_k)\leq\frac{2t-f(x)}{k-f(x)}.
\]

\end{example}

So, Lyapunov 1-forms for $(X,U)$ can give us information about the behavior of the paths of $X$ outside the invariant set.
Now we want to prove the following theorem which is similar to Theorem 2 of \cite{Farber} and study the existence of these forms.

\begin{definition}
 Let $X$ be an $L$-diffusion in $M$ and $W$ be a closed subset in $M$ invariant under $X$. We say that a measure in $\mathcal{M}_L$ is coherent relative to $W$ if there exist an open neighborhood $U$ of $W$ such that $\mu(U)=0$. Denote by $\mathcal{M}_L^W$ the set of coherent measures.
\end{definition}

\begin{theorem}
 Let $X$ be a diffusion in $M$, $W$ be a closed invariant set for $X$. Then

\begin{itemize}
 \item [i-] if $\mathcal{M}^W_L\not=\emptyset$ and there is a non trivial cohomology class $\xi$ in
$H^1(W^C)$, there is a Lyapunov 1-form $\beta\in\xi$ for $(X,W)$ if and only if for any $\mu\in\mathcal{M}^W_L$, we have
\[
 \mathcal{J}^L_\mu(\beta)<0~\textrm{for all}~\beta\in\xi.
\]
\item [ii-] if $\mathcal{M}^W_L=\emptyset$ there exists Lyapunov 1-forms for $(X,W)$ for any $\xi$ in
$H^1( W^C)$.
\end{itemize}

\end{theorem}

\begin{proof}\begin{itemize}
              \item[i-] If there exists a Lyapunov 1-form $\beta$ for $(X,W)$ in the class $\xi$ then, using that $\mu\in\mathcal{M}^W_L$ satisfy
$\mu(W)=0$ and the definition of Lyapunov 1-form we obtain
\[ \mathcal{J}^L_\mu(\beta)=\int_{W^C}\s\beta(L)~\mu<0,\]
as desired.

On the other side we assume that for any $\beta\in\xi$ we have that $ \mathcal{J}^L_\mu(\beta)<0$ and fix some
$\beta\in\xi$.
We use a similar argument to the one used by Latschev in \cite{Latschev}.
Consider the following sets
%
\[
 \mathcal{C}_{\beta}=\{f=\s\beta(L)+Lg,~g\in \mathcal{D}(L)\cap C^2(W^C)\}
\]
\[
 \mathcal{C}_{W}=\{f\in C^0(M), ~f<0~\textrm{in}~W^C\},
\]
where $\mathcal{D}(L)$ is the set of functions $g$ such that  $Lg$ is continuous in $M$, and assume that $\mathcal{C}_{\beta}\cap \mathcal{C}_{W}=\emptyset$
. Then, the Hahn-Banach theorem (see, for
example, Rudin \cite[pg. 59]{rudin}) implies 
the existence of a signed measure $\mu$ such that
\begin{eqnarray*}
 1)&& \int_M f~\mu<0~\textrm{if}~f\in \mathcal{C}_{W}\hspace{.5cm}\textrm{and}\hspace{.5cm}\\
2)&&\int_M f~\mu
\geq0~\textrm{if}~f\in \mathcal{C}_{\beta}.
\end{eqnarray*}
Equation 1) says that $\mu$ is a positive measure with support contained in $W^C$. Using that $\mathcal{C}_{\beta}$ is affine we get that equation 2) implies
\[
 \int_M \s\beta(L)~\mu=\int_M f~\mu\geq0
\]
for any $f\in \mathcal{C}_\beta$ or, equivalently, $\int Lg~\mu=0$ for any $ g\in C^2(W^C)$.

Let $\rho:M\rightarrow [0,1]$ be a smooth function such that $\rho=1$ in a neighborhood of $W$ and $\rho=0$ in a neighborhood of $\textrm{supp}(\mu)$ then for any $f\in C^2(M)$ we get that
\begin{eqnarray*}
 \int_ML(f)~\mu
&=&\int_ML(\rho f+(1-\rho)f)~\mu\\
&=&\int_ML(\rho f)~\mu+\int_ML((1-\rho)f)~\mu\\
&=&0.
\end{eqnarray*}
because $(1-\rho)f$ is 0 in a neighborhood of $W$. Then $\mu\in\mathcal{M}^W_L$ such that
$$\int_M \s\tilde \beta(L)~\mu\geq 0,$$
contradicting the hypothesis.
Therefore $ \mathcal{C}_\beta\cap
\mathcal{C}_{W}\not=\emptyset$. So there is a Lyapunov 1-form
in $\xi$.

\item[ii-] Given $\beta$ in some $\xi$ we define $\mathcal{C}_{W}$ 
and $\mathcal{C}_{\beta}$ as above. Assume that $ \mathcal{C}_\beta\cap
\mathcal{C}_{W}=\emptyset$. 
Doing the same calculations we will find a non zero measure $\mu\in\mathcal{M}_L^W$ which is a contradiction. So there is a  a Lyapunov 1-form in $\xi$.

             \end{itemize}

\end{proof}

\begin{corollary}
Assume that there is a Lyapunov 1-form $\beta$ for $(W,X)$ in a non trivial cohomology class $\xi\in H_1(W^C)$. Then for any
1-cycle
$\rho_\mu^L$ associated by $\mathcal{J}^L_\mu$ to some $\mu\in\mathcal{M}_L$ with support
contained in $W^C$  we have that
\[<\xi,\rho_\mu^L>< 0.\]
\end{corollary}

A point $x$ is said to be recurrent if for any neighborhood $U$ of $x$ we have that
\[
\mP\left[\left\{\omega\in \Omega,~\exists~\{t_j\}\uparrow\infty~\textrm{s.t.}~X_{t_j}^x(\omega)\in  U,~j\in
\mn\right\}\right]=1
\]

Clearly,  given an $L$-diffusion we can split the manifold $M$ into two subsets, the set $\mathcal{R}$ of recurrent
points and the set of transient points $\mathcal{T}$. It is known that the support of any invariant measure is a closed set that lies
in the closure of $\mathcal{R}$.

\begin{corollary}
 Consider $(\overline{\mathcal{R}},X)$ where $\mathcal{R}$ is the invariant set of recurrent points, then there is a function $f:M\rightarrow\mr$ such that $Lf$ is continuous and $Lf|_{\overline{\mathcal{R}}^C}<0$.

\end{corollary}

\end{document}